\documentclass[a4paper,12pt]{article}

\usepackage{graphicx}
\usepackage{multicol,multirow}
\usepackage{amsmath,amssymb,amsfonts}
\usepackage{mathrsfs}
\usepackage{amsthm}
\usepackage{rotating}
\usepackage{appendix}
\usepackage[numbers]{natbib}
\usepackage{ifpdf}
\usepackage{xcolor}
\usepackage[colorlinks,allcolors=blue]{hyperref}
\usepackage{ulem}
\usepackage{lipsum}
\usepackage[top=1in, bottom=1in, left=0.5in, right=0.5in]{geometry}
\usepackage{microtype}

\usepackage{mathtools} 
 \usepackage{amsmath} 

\usepackage{xcolor}

\newtheorem{theorem}{Theorem}[section]
\newtheorem{lemma}[theorem]{Lemma}
\newtheorem{corollary}[theorem]{Corollary}

\newtheorem{conjecture}[theorem]{Conjecture}

\newtheorem{definition}[theorem]{Definition}

\theoremstyle{cupremark}
\newtheorem{remark}[theorem]{Remark}

\renewenvironment{proof}{{\bfseries Proof.} }{\qed}

\newtheorem*{theorem-rep}{\theoremname}  
\newtheorem*{corollary*}{\theoremname}
\newtheorem*{lemma*}{Lemma}

\theoremstyle{cupproof}

\numberwithin{equation}{section}
\title{Seifert fibered 3-manifolds and Turaev-Viro invariants volume conjecture}
\author{Shashini Marasinghe}
\date{}

\begin{document}

\maketitle

\abstract{We study the large $r$ asymptotic behaviour of the Turaev-Viro invariants of oriented Seifert fibered 3-manifolds at the root $q=e^\frac{2\pi i}{r}$. As an application, we prove the volume conjecture for large families of oriented Seifert fibered 3-manifolds with empty or non-empty boundary.}

\section{Introduction}

 \quad \quad In 1988, Witten \cite{witten1989quantum} proposed an extension of the Jones polynomial of links in 3-sphere to arbitrary 3-manifolds, and conjectured a connection between quantum invariants and classical invariants of links and 3-manifolds. This approach led to the discovery of a family of quantum invariants for closed-oriented 3-manifolds with an embedded colored link. Subsequently, Reshetikhin and Turaev \cite{reshetikhin1991invariants} constructed 3-manifold and link invariants that satisfy the topological quantum field theory (TQFT) properties expected by Witten's theory. In this paper, we are concerned with the $SU(2)$-Witten-Reshetikhin-Turaev invariants. For any 3-manifold $M$, any integer $r\geq3$, and a primitive $4r$-th root of unity $e^\frac{\pi i}{2r}$, this theory produces a complex-valued invariant $RT_r(M,e^\frac{\pi i}{2r})$. For Seifert fibered 3-manifolds with fibering over orientable 2-orbifold base, the $RT_r(M,e^\frac{\pi i}{2r})$ invariants were computed in \cite{lawrence1999witten}. For the general case, orientable or non-orientable 2-orbifold base, see \cite{hansen2001reshetikhin}. 

The Turaev-Viro invariants \cite{turaev1992state} were introduced in the early 1990s as a state sum for a triangulation of 3-manifolds, offering a combinatorial framework for computing 3-manifold invariants. Turaev-Viro invariants are indexed by integers $r$, which depend on the choice of a root of unity. In this paper, we focus on $SO(3)$-Turaev-Viro invariants. For any 3-manifold $M$, odd integers $r\geq3$, and $2r$-th root of unity $q=e^\frac{2\pi i}{r}$, these invariants are denoted by $TV_r'(M,q)$. Moreover, Turaev-Viro invariants are real-valued and equal to the squared modulus of the Witten-Reshetikhin-Turaev invariants under suitable conditions \cite{robertproof}.

This paper aims to prove that some families of oriented Seifert fibered 3-manifolds with orientable or non-orientable 2-orbifold base satisfy the generalized Turaev-Viro invariants volume conjecture stated by Detcherry-Kalfagianni in \cite{detcherry2020gromov}. This conjecture is the 3-manifold analogue of the volume conjecture proposed by Kashaev in \cite{kashaev1997hyperbolic}, and later by Murakami-Murakami in \cite{murakami2001colored}. In \cite{chen2018volume}, Chen-Yang stated the version of this conjecture for hyperbolic 3-manifolds. The Chen-Yang conjecture has been proved for several classes of hyperbolic 3-manifolds. This includes families of link complements in the 3-sphere \cite{detcherry2018turaev,Ku}, and closed 3-manifolds obtained by Dehn filling along such links \cite{Oh,WongYang,CZ,GMWY}, and link complements in other 3-manifolds \cite{belletti2022growth,Bel}.
Partial verifications of the generalized  Tuarev-Viro invariants volume conjecture were given in \cite{detcherry2020gromov,Wong,KuM1,KuM2}.
For Seifert fibered 3-manifolds, Detcherry-Kalfagianni \cite{detcherry2020gromov} proved that the Tuarev-Viro invariants $TV_r'(M,q)$ grow at most polynomially with respect to $r$.
    In this paper, we prove that for large families of Seifert fibered 3-manifolds with empty or non-empty boundary, the growth of $TV_r'(M,q)$, with respect to odd integer $r\geq3$, is bounded below by 1, hence verifying the generalized volume conjecture for these classes of 3-manifolds.

Notice that throughout this paper, we focus on odd integer $r\geq 3$ and $q$ that is  $2r$-th root of unity with $q=e^\frac{2\pi i}{r}$, unless mentioned otherwise. Here, $TV_r$ denotes the $SU(2)$ version of the Turaev-Viro invariants, 
while $TV_r'$ denotes the $SO(3)$ version. Similarly, $RT_r$ represents 
the $SU(2)$ version of the Witten-Reshetikhin-Turaev invariants, 
and $RT_r'$ represents the $SO(3)$ version. We will use the following definition to state the Turaev-Viro invariants volume conjecture.   

\begin{definition}[Section 1.1 \cite{detcherry2020gromov}]
For a compact oriented 3-manifold $M$, let $TV_r'(M,q)$ denote the $r$-th Turaev-Viro invariant of M at $q=e^\frac{2\pi i}{r}$. Then, the Turaev-Viro growth rate is 
    \[LTV(M)=\underset{r \xrightarrow{}\infty}{\text{lim sup }}\frac{2\pi}{r}\text{log}|TV_r'(M,q)|,\]

    where $r$ runs over all odd positive integers.
    
\end{definition}

Recall that the notation $||M||$ denotes the Gromov norm (a.k.a. simplicial volume) of an oriented 3-manifold $M$. Now, we state the volume conjecture.

\begin{conjecture}[Turaev-Viro invariants volume conjecture 8.1 \cite{detcherry2020gromov}]
\label{con:TV volume}
 For every compact orientable 3-manifold $M$ with an empty or toroidal boundary, we have 
 \[LTV(M)=\underset{r \xrightarrow{}\infty}{\text{lim sup }}\frac{2\pi}{r}\text{log}|TV_r'(M,q)|=v_3||M||,\]
 where $r$ runs over all odd positive integers and $v_3\simeq 1.0149$ is the volume of a regular ideal tetrahedron.
     
 \end{conjecture}

Let $M$ be an oriented Seifert fibered 3-manifold with empty or non-empty boundary that fibers over a genus $g$ 2-orbifold base with the symbol $(\epsilon,g; (a_1, b_1), \ldots, (a_n, b_n))$. The invariant pair $(a_j,b_j)$ represents the $j$-th exceptional fiber. Recall that, $ n \geq 0$, and $g>0$ are integers, and for each $j = 1, \ldots, n$, $a_j \geq 1$ and $b_j$ are coprime integers. For an orientable 2-orbifold base, denote $\epsilon=O$; for a non-orientable 2-orbifold base, denote $\epsilon =N$. The genus for the non-orientable 2-orbifold base is the number of $\mathbb{R}P^2$ connected sums. Also, $b_j^*$ denotes congruence inverse of $b_j$ modulo $a_j$; that is $b_j b_j^* \equiv 1 (\text{mod }a_j)$. Additionally, see Section \ref{section 3} for a detailed explanation of the construction of a closed-oriented Seifert fibered 3-manifold \( D(M) \).

   Throughout this paper, all computations are carried out under the assumptions that $a_j\geq3$ and gcd$(a_j,b_j)=1$ for all $j=1\ldots,n$. Now, we are ready to state the main theorem of this paper.

\begin{theorem}
\label{main.A}
    Let $M$ be an oriented Seifert fibered 3-manifold with boundary, described by the symbol 
\[
(\epsilon,g; (a_1, b_1), \ldots, (a_n, b_n)),
\] where each $a_j\geq3$ and gcd$(a_j,b_j)=1$ for all $j=1\ldots,n$. Suppose that there is an integer $\gamma > 0$ and $\boldsymbol{\mu}=(\mu_1,\ldots ,\mu_n)$ with $\mu_j=\pm 1$ such that \[\gamma+\mu_jb_j^* \equiv 0 (\text{mod }a_j)\] for $j=1,\ldots,n$.
Then, $M$ and $D(M)$ satisfy Conjecture \ref{con:TV volume}.

That is, $LTV(M)=LTV(D(M))=0$. \end{theorem}

The following corollary gives large families of examples of oriented Seifert fibered 3-manifolds that satisfy the condition of Theorem \ref{main.A}.

\begin{corollary}
\label{coro A}
    Let $M$ be an oriented Seifert fibered 3-manifold with boundary, described by the symbol 
\[
(\epsilon,g;(a_1, b_1), \ldots, (a_n, b_n))
,\] where each $a_j\geq3$ and gcd$(a_j,b_j)=1$ for all $j=1\ldots,n$. If either
\begin{enumerate}
    \item [(a)] $a_1,\ldots , a_n$ are coprime; or
    \item [(b)] $a_1=\ldots =a_n=a$ and there are $\mu_1,\ldots,\mu_n$ such that $\mu_1 b_1^* \equiv \ldots \equiv \mu_n b_n^* \pmod{a}$,
\end{enumerate} then $M$ and $D(M)$ satisfy Conjecture \ref{con:TV volume}.
\end{corollary}

The next corollary discusses other large families of 3-manifolds that satisfy Theorem \ref{main.A}.

\begin{corollary}
\label{exam 3}

 Suppose $M$ and $D(M)$ are 3-manifolds in the statement of Theorem \ref{main.A}. Let $L$ be a link in $M$ or $D(M)$ that has simplicial volume (a.k.a. Gromov norm) zero. Then, Conjecture \ref{con:TV volume} is true for $M\backslash L$.

\end{corollary}

The outline of the proof of Theorem \ref{main.A} and the organization of this paper are as follows: 
In Section \ref{section 2}, we discuss preliminaries. Section \ref{section 3} begins with the proof of Lemma \ref{lemma}, which is essential to establish Theorem \ref{main.A}. For this, we first utilize the formula of Witten-Reshetikhin-Turaev invariants \( RT_r(\_~,e^{\frac{\pi i}{2r}}) \) for closed-oriented Seifert fibered 3-manifolds, as established in \cite{hansen2001reshetikhin}. We also examine the relationship between \( TV_r(M,  e^{\pi i / r}) \) and \( RT_r(D(M), e^{\pi i /2 r}) \), as stated in Theorem \ref{TV WRT con A} for Seifert fibered 3-manifolds with boundary, along with the corresponding relationship outlined in Theorem \ref{TV WRT con B} for closed-oriented Seifert fibered 3-manifolds. Then, by using Lemma \ref{tv lemma} we compute \( TV_r'(M,  e^{2\pi i / r}) \) and \( TV_r'(D(M),  e^{2\pi i / r}) \). This analysis allows us to prove Lemma \ref{lemma}. Following this, we prove  Theorem \ref{main.A}. Finally, in Section \ref{section 4}, we discuss example families where Theorem \ref{main.A} is satisfied.

\subsection{Acknowledgment}

I want to express my sincere gratitude to Professor Efstratia Kalfagianni for her invaluable support and guidance throughout the development of this paper. Her insightful advice and encouragement were instrumental in shaping my work. I also extend my appreciation to Professor Francis Bonahon for taking the time to listen to my presentation and for providing thoughtful feedback during its early stages. This research was partially supported in the form of a graduate Research Assistantship by NSF Grant DMS-230433.

\section{Preliminaries}
\label{section 2}

\subsection{Oriented Seifert fibered 3-manifolds}

Throughout this paper, we focus on oriented Seifert fibered 3-manifolds. The reader is referred to the citations \cite{seifert89topology,hatcher2007notes,jankins1983lectures} for more information.

An oriented Seifert fibered 3-manifold $M$ is a 3-dimensional manifold that can be decomposed into a collection of disjoint circles, called \textit{fibers}, arranged in a specific way. Each fiber in $M$ has a neighborhood that is homeomorphic to a fibered solid torus, denoted by $T(a, b)$ where $a\geq 1$ and $b$ are coprime integers.  The fibered solid torus $T(a, b)$ is constructed as the quotient space:
\[
T(a, b) = \big(D \times [0, 1]\big) / \sim,
\]
where the equivalence relation is $(x, 1) \sim \big(\pi(x), 0\big)$ for $x \in D$. Here, $\pi: D \to D$ denote the counter-clockwise rotation of the unit disk $D \subset \mathbb{C}$ by an angle of $\frac{2\pi b}{a}$. The image of the line segment $\{0\} \times [0, 1]$ (the central axis of the cylinder) in $T(a, b)$ is called the \textit{middle fiber}. If $a =  1$, the middle fiber is called an \textit{ordinary fiber}, whose neighborhood resembles a standard solid torus. If $a > 1$, the middle fiber is called a \textit{singular} or an \textit{exceptional fiber}. Additionally, the orientation of $T(a,b)$ is determined by the orientation of $D$ and $[0,1]$.

The \textit{2-orbifold base} $\Sigma$ of the Seifert fibered 3-manifold $M$ is the quotient space of $\Sigma \times S^1$ by identifying each fiber to a point.
To visualize the construction, let $\Sigma$ be a genus $g$ surface with $n$ punctures (boundary components). Remove $n$ solid tori $D_j\times S^1$ from $\Sigma\times S^1$ where $D_j$ for $j=1,\ldots,n$, represent neighborhood disks of $n$ punctures. Now, glue them back in along the boundary $\partial D_j\times S^1$ of the removed tori such that the meridian of each torus, represented by a loop in $\partial D_j \times S^1$ that bounds a disk in $D_j \times S^1$, is mapped to a curve with slope $\frac{b_j}{a_j}$ on the torus where integer $a_j\geq1$ coprime to $b_j$ for all $j=1,\ldots,n$. The 2-orbifold base of $M$ is genus $g$ with $n$ cone points (singularities) that arose due to collapsing each exceptional fiber to a point on the 2-orbifold base.

The following summarizes some of the properties.

\begin{itemize}
    \item The symbol of oriented Seifert fibered 3-manifold $M$ is 
    $(\epsilon,g;(a_1,b_1),\ldots,(a_n,b_n))$, where $n\geq 0$ and $g>0$ are integers, and  $a_j \geq 1$ and $b_j$ are coprime pairs of integers for $j=1,\ldots,n$. Each pair $(a_j,b_j)$ represents an exceptional fiber. For an orientable 2-orbifold base, denote $\epsilon=O$; for a non-orientable case, denote $\epsilon=N$. 
    \item The Euler characteristic of the 2-orbifold base is
    \[2-2g-\sum_j(1-1/a_j).\]
    \item The Seifert Euler number of the Seifert fibration of $M$ is \[ e(M)\coloneqq-\sum_jb_j/a_j.\]

    \item Two Seifert fibered 3-manifolds are equivalent if there is a fiber and orientation preserving homeomorphism between them.

    \item The following operations can be applied to an invariant pair while preserving the isomorphism class of the corresponding Seifert fibration:
    \begin{itemize}
        \item Addition or deletion of any invariant pair $(a,b)=(1,0)$.
        \item Replace any $(0,\pm 1)$ by $(0,\mp 1)$.
        \item  Replace each $(a_j,b_j)$ by $(a_j,b_j+k_ja_j)$ where $k_j\in\mathbb{Z}$ with the additional condition that $\sum k_j=0$ if $\partial M=\emptyset$.
    \end{itemize}
    
\end{itemize}

\subsection{Turaev-Viro invariants }

This section provides a brief overview of the construction of the Turaev-Viro invariants. We focus on the case where $r$ is odd and adopt the $SO(3)$-version for all the constructions and results, as was done in \cite{detcherry2020gromov}. For further details, the reader is referred to the citations \cite{turaev1992state,kauffman1994temperley}.

For notation, we denote by $TV_r$ the $SU(2)$-version of the Turaev-Viro invariants for integers $r \geq 3$, computed at a $4r$-th root of unity, and by $TV_r'$ the $SO(3)$-version, computed at a $2r$-th root of unity for odd $r\geq3$.

\subsubsection{Notation} 

Let $r\geq3$ be an odd integer and $q=e^\frac{2\pi i}{r}$. Let $I_r$ be the set where the elements are all non-negative, even integers less than $r-2$. That is, $I_r=\{0,2,4,\ldots,r-3\}$.

Now, we define some important notations and definitions for constructing Turaev-Viro invariants.

\begin{itemize}

    \item Denote by $\{n\}$ the \textit{quantum integer}  \[\{n\}=q^n-q^{-n}=2\text{sin}(\frac{2\pi}{r})[n]\] 
where $[n]=\frac{q^n-q^{-n}}{q-q^{-1}}=\frac{2\text{sin}(\frac{2n\pi}{r})}{2\text{sin}(\frac{2\pi}{r})}$.
    
    \item Denote by $\{n\}!$ the \textit{quantum factorial} \[\{n\}!=
\prod_{i=1}^n\{i\}.\]

    \item A triple $(i,j,k)$ is an \textit{admissible triple} of elements in $I_r$ if it satisfies the inequalities $i\leq j+k$,  $j\leq i+k$, $k\leq i+j$, and $i+j+k\leq 2(r-2)$.
    
    \item For an admissible triple $(i,j,k)$, define 
    \[\Delta (i,j,k)\coloneqq\zeta_r ^{\frac{1}{2}}\left(\frac{\left\{\frac{i+j-k}{2}\right\}!\left\{\frac{i+k-j}{2}\right\}!\left\{\frac{j+k-i}{2}\right\}!}{\left\{\frac{i+j+k}{2}+1\right\}!}\right)^{\frac{1}{2}}\] 
    where $\zeta_r=2\text{sin}(\frac{2\pi}{r})$.

    \item A 6-tuple $(i,j,k,l,m,n)$  of elements in $I_r$ is called \textit{admissible} if the triples $F_1=(i,j,k)$, $F_2=(j,l,n)$, $F_3=(i,m,n)$, and $F_4=(k,l,m)$ are admissible.

    \item For an admissible 6-tuple $(i,j,k,l,m,n)$, define \textit{quantum 6$j$-symbol} at the root $q$ as follows.

\[
\left| 
\begin{array}{ccc}
i & j & k \\
l & m & n \\
\end{array}
\right|
= (\zeta_r)^{-1}(\sqrt{-1})^{\lambda}\prod_{k=1}^4\Delta(F_{k})\sum_{z=\text{max}\{T_1,T_2,T_3,T_4\}}^{\text{min}\{Q_1,Q_2,Q_3\}}\frac{(-1)^z\{z+1\}!}{\prod_{b=1}^4\{z-T_b\}!\prod_{c=1}^3\{Q_c-z\}!}\]

where $\lambda =i+j+k+l+m+n$, and

\[\begin{split}
    T_1&=\frac{i+j+k}{2},\hspace{15pt} T_2=\frac{i+m+n}{2},\hspace{15pt} T_3=\frac{j+l+n}{2},\hspace{15pt} T_4=\frac{k+l+m}{2},\\
    Q_1&=\frac{i+j+l+m}{2},\hspace{15pt} Q_2=\frac{i+k+l+n}{2},\hspace{15pt} Q_3=\frac{j+k+m+n}{2}.
\end{split}
\]

\end{itemize}

Let $M$ be a compact orientable 3-manifold and consider $\mathcal{T}$ as a triangulation. In the case of $\partial M\neq \emptyset$, $\mathcal{T}$ is a (partially) ideal triangulation. That is, it consists of finitely many truncated Euclidean tetrahedra glued together along their faces, where the triangles arising from the truncations form a triangulation of $\partial  M$.

The assignment of an element of $I_r$ to the edges of $\mathcal{T}$ is known as \textit{coloring at level $r$} of a triangulated 3-manifold and denoted by $(M,\mathcal{T})$. Such a coloring is called \textit{admissible} if the 6-tuple assigned to the edges of tetrahedrons of $(M,\mathcal{T})$ satisfies the admissibility conditions.

Let $c$ be an admissible coloring of $(M,\mathcal{T)} $ at level $r$. 
Let \[|e|_c=(-1)^{c(e)}[c(e) +1],\] for a coloring $c$ assigned to an edge $e\in E$. For each tetrahedron $\Delta$  in $(M,\mathcal{T)} $ with coloring $c$, $ |\Delta |_c$ is the quantum 6$j$- symbol associated with the admissible 6-tuple assigned to $\Delta$ by $c$.

  Let \(V\) denote the set of interior vertices of \(\mathcal{T}\) that do not lie on \(\partial M\), and let \(E\) denote the set of edges of \(\mathcal{T}\), excluding those arising from the truncation of vertices.

 Now, we are ready to define the $SO(3)$-version of Turaev-Viro invariants using the state sum model.

 \begin{definition}[Theorem 3.2 \cite{detcherry2020gromov}]
     Let $A_r(\mathcal{T})$ be the set of $SO(3)$-admissible coloring of $(M,\mathcal{T})$ at level $r$, and let $b_2$ and and $b_0$ denote the second and zeroth $\mathbb{Z}_2$-Betti number of $M$. Define $SO(3)$-version of $r$-th Turaev-Viro invariant as
     \[TV_r'(M,q)=2^{b_2-b_0}(\eta ')_r^{2|V|}\sum_{c\in A_r(\mathcal{T})}{\prod_{e\in E}|e|_c \prod_{\Delta\in T}|\Delta|_c}\]
     where $\eta ' =\frac{2\text{sin}(\frac{2\pi}{r})}{\sqrt{r}}.$

 \end{definition}

 Next, we discuss an important Lemma that helps us to obtain the required $SO(3)$-version $TV_r'$ at $q$ in the proof of Theorem \ref{main.A}. The general proof for any 3-manifold can be found in Theorem~2.8 of \cite{detcherry2018turaev}, while proofs for certain special cases involving oriented Seifert fibered manifold with oriented 2-orbifold base are given in Lemma~3.2 of \cite{liu}.

\begin{lemma}[Theorem 2.8 \cite{detcherry2018turaev}]
\label{tv lemma}
Let $M$ be any $3$–manifold, and let $b_2$ and $b_0$ denote the second and zeroth $\mathbb{Z}_2$–Betti numbers of $M$, respectively. Suppose $r\geq3$ is an odd integer. Then,  

\begin{itemize}
    \item[1] $TV_r(M,e^{\frac{2\pi i}{r}})$ is Galois conjugate to $TV_r(M,e^{\frac{(r-1)\pi i}{r}})$. The same statement holds for $TV_r'$ in place of $TV_r$.  

    \item[2] $TV_r(M,e^{\frac{2\pi i}{r}}) \;=\; TV_3(M,e^{\frac{2\pi i}{3}})\,\cdot\, TV_r'(M,e^{\frac{2\pi i}{r}})$.  

    \item[3] $TV_r(M,e^{\frac{\pi i}{r}}) \;=\; TV_3(M,e^{\frac{\pi i}{3}})\,\cdot\, TV_r'(M,e^{\frac{(r-1)\pi i}{r}})$.  

    \item[4] If $\partial M = \emptyset$, then 
\[
TV_3(M,e^{\frac{2\pi i}{3}}) = 2^{\,b_2 - b_0},
\]
and if $\partial M \neq \emptyset$, then 
\[
TV_3(M,e^{\frac{2\pi i}{3}}) = 2^{\,b_2}.
\]

\end{itemize}
\end{lemma}

\subsection{Witten-Reshetikhin-Turaev invariants}

Unlike the previous subsection, the constructions and the results presented in this subsection are based on the $SU(2)$-version and denoted by $RT_r(\_~,e^\frac{\pi i}{2r})$ for $4r$-th root of unity $e^\frac{\pi i}{2r}$ and any integer $r\geq3$. In contrast, denote the $SO(3)$-version by $RT_r'(\_~,e^\frac{\pi i}{r})$ where $e^{\frac{\pi i}{r}}$ is a primitive $2r$-th root of unity and $r \geq 3$ is an odd integer.

 The Witten-Reshetikhin-Turaev TQFT is a functor from the category of (2+1)-dimensional cobordisms, where objects are closed-oriented 2-manifolds, and morphisms are cobordisms between compact, oriented 3-manifolds with boundaries to the category of finite-dimensional $\mathbb{C}$-vector spaces. Thus, for any oriented compact closed 3-manifold $M$, $RT_r(M,e^\frac{\pi i}{2r})\in \mathbb{C}$.

 The construction of Witten-Reshetikhin-Turaev invariants through a skein theoretic approach is found in \cite{blanchet1995topological}. The following theorem discusses summarized properties of the $r$-th Witten-Reshetikhin-Turaev invariant $RT_r(M,e^\frac{\pi i}{2r}):=RT_r(M)$ of 3-manifold $M$ at $e^\frac{\pi i}{2r}$ found in \cite{blanchet1995topological}.

\begin{theorem}[\cite{kirby19913}, \cite{blanchet1995topological}]
Let $M$ and $N$ be connected oriented 3-manifolds, and $\overline{M}$ is $M$ with reversed orientation.
   \begin{enumerate}
       \item $RT_r(\overline{M})=\overline{RT_r(M)}$.
       \item $RT_r(M\# N)=RT_r(M)\cdot RT_r(N)$.
       \item $RT_r(S^2\times S^1)=\frac{\sqrt{r/2}}{\text{sin}(\frac{\pi}{r})}$ and $RT_r(S^3)=1$.

       \item For any closed, compact, and oriented surface, $\Sigma$, $RT_r(\Sigma)$ is a finite-dimensional $\mathbb{C}$-vectors space (Hermitian vector space) denoted by $V_r(\Sigma)$, and $V_r(\emptyset)=\mathbb{C}  $. Moreover, for disjoint union $V_r(\Sigma_1\sqcup \Sigma_2)\simeq V_r(\Sigma_1)\otimes V_r(\Sigma_2)$ and $V_r(\overline{\Sigma})=\overline{V_r(\Sigma)}$.

       \item For any compact, oriented 3-manifold $M$ with $\partial M \simeq \Sigma$, there is an associated vector $RT_r(M)\in V_r(\Sigma)$. For disjoint union $M=M_1\sqcup M_2$, we have $RT_r(M)=RT_r(M_1)\otimes RT_r(M_2)\in V_r(\Sigma_1)\otimes V_r(\Sigma_2) .$
       
   \end{enumerate}
    
\end{theorem}

\subsection{Witten-Reshetikhin-Turaev invariants of oriented Seifert fibered 3-manifolds}

We use the following derivation of the $r$-th $SU(2)$-invariant $RT_r(M,e^\frac{\pi i}{2}r)$ of a closed, oriented Seifert fibered 3-manifold $M$ to prove Theorem \ref{main.A}.

\begin{lemma}[Theorem 8.4 \cite{hansen2001reshetikhin}]
\label{lemma seifert}
Let $M$ be closed, oriented Seifert fibered manifold with symbol $(\epsilon,g;(a_1,b_1), \ldots ,(a_n,b_n))$, where $n\geq 0$ and $g>0$ are integers, and where $a_j \geq 1$ and $b_j$ are coprime pairs of integers for $j=1,\ldots,n$. The Seifert Euler number of the Seifert fibration is 
    \[e(M)=-\sum_j b_j/a_j.\]
   Then,
   \[\begin{split} RT_r(M,e^{\frac{\pi i}{2r}})=&e^{\frac{\pi i}{2r}{[3(a_\epsilon -1)\text{sgn}(e(M))-e(M)-12 \sum_{j=1}^n s(b_j,a_j)}]}\\
                                              &\times (-1)^{a_\epsilon g}\frac{i^nr^{a_\epsilon g/2-1}}{2^{n+a_\epsilon g/2-1}\sqrt{\prod_ja_j}}e^{\frac{3\pi i}{4}(1-a_\epsilon)\text{sgn}(e(M))}{Z_{(\epsilon,r)}(M,e^\frac{\pi i}{2r})}\end{split}\]
\[Z_{(\epsilon,r)}(M,e^\frac{\pi i}{2r}) = \sum_{(\gamma, \boldsymbol{\mu})} \left \{\frac{(-1)^{\gamma a_\epsilon g}e^{\frac{\pi e(M) \gamma^2i}{2r}}\prod_{j=1} ^{n} \left(\mu_j e^{ {  \frac{ - \pi \gamma\mu_j i}{a_j r}}}\right)}{\mathrm{sin}^{n+a_\epsilon g-2}(\pi \gamma / r)}  \cdot \sum_{\boldsymbol{m}}\prod_j e^ {- \left(\frac{2 \pi m_j(\gamma + \mu_jb_j^*)}{a_j}+\frac{2\pi r m_j^2 b_j^*}{a_j} \right)\cdot i} \right \} \]

   where $\epsilon=O$ for an oriented 2-orbifold base and $\epsilon =N$ for a non-orientable base. Also, let $a_{O} =2$ and $a_{N}=1$.

                    Here, $j$ ranges over \{$1, \ldots , n$\} and $(\gamma,\boldsymbol{\mu},\boldsymbol{m})=(\gamma,(\mu_1,\ldots,\mu_n),(m_1,\ldots,m_n))$ ranges over $\{1,2,\ldots,r-1\} \times \{\pm 1\}^n \times \mathbb{Z}/{a_1\mathbb{Z}}\times \ldots \times \mathbb{Z}/{a_n\mathbb{Z}}. $

   The notation $b_j^*$ denotes any congruence inverse of $b_j$ modulo $a_j$, namely $b_j b_j^* \equiv 1 (\text{mod }a_j)$, $\mathrm{sgn}(e(M))$ denotes the sign of $e(M)$, with value $\pm 1$ or $0$; $\mathrm{s}(b_j,a_j)$ denotes the Dedekind sum $(4a)^{-1} \cdot \sum_{l\in \{1,2,\ldots,a_j-1\}}\mathrm{cot}(\pi l/a_j)\mathrm{cot}(\pi l b_j/a_j).$
   
\end{lemma}  

\begin{remark}
    The formula for $RT_r(\_~,e^\frac{\pi i}{2r})$ of oriented Seifert fibered 3-manifold with orientable 2-orbifold base is discussed in \cite{lawrence1999witten}, and it coincides with our Lemma \ref{lemma seifert}.
\end{remark}

\subsection{Relationship between Turaev-Viro invariants and Witten-Reshetikhin-Turaev invariants}

We use Theorem \ref{TV WRT con A} and Theorem \ref{TV WRT con B} to establish the relationship between Turaev-Viro invariants and Witten-Reshetikhin-Turaev invariants for integers $r\geq 3$ under $SU(2)$-setting. The proofs of these theorems for closed manifolds are due to \cite{roberts1995skein,turaev2010quantum}, and for manifolds with a non-empty boundary for $SU(2)$-setting are due to \cite{robertproof,roberts1995skein}, and $SO(3)$-setting can be found in \cite{detcherry2018turaev}.

\begin{theorem}[\cite{robertproof}]
\label{TV WRT con A}
    Let $M$ be a 3-manifold with boundary, and $r\geq3$ be an integer. Then,
   
     \[TV_r(M,e^\frac{\pi i}{r})=\eta_r^{-\chi(M)}RT_r(D(M),e^{\frac{\pi i}{2r}}),\]

    where $\chi(M)$ is the Euler characteristic of $M$ and $\eta_r=\frac{2\text{sin}(\frac{2\pi}{r})}{\sqrt{2}r}.$
\end{theorem}

The following is the closed 3-manifold version for Theorem \ref{TV WRT con A} and can be easily derived from Theorem \ref{TV WRT con A} using the properties of Hermitian form of $D(M)$.

\begin{theorem}[\cite{robertproof}]
\label{TV WRT con B}
    For $M$ an oriented compact 3-manifold with empty or toroidal boundary and $r\geq 3$ an integer, we have
    \[TV_r(M,e^\frac{\pi i}{r})=||RT_r(M,e^\frac{\pi i}{2r})||^2.\]
\end{theorem}

Moreover, we use the following results obtained for Seifert fibered 3-manifolds in Sections Three, and Four to establish Conjecture \ref{con:TV volume}. Note that we have adjusted the notation 
to be consistent with our conventions: although Conjecture \ref{con:TV volume} uses 
$TV_r$ for the $SO(3)$ version, here we use $TV_r'$ to align with the notation introduced in the Introduction and at the beginning of this section. 

\begin{theorem}[Theorem 6.2 \cite{detcherry2020gromov}]
\label{1}
    Let $M$ be a compact, orientable 3-manifold that is Seifert fibered. Then, there exist constants $A>0$ and $N>0$, depending on $M$, such that
    \[TV_r'(M,q) \leq Ar^N.\]
    
    Thus, $LTV(M)\leq 0 $.
    
\end{theorem}

\begin{theorem}[Theorem 1.1 \cite{detcherry2020gromov}]
\label{2}
    There exists a universal constant $C > 0$ such that for any compact orientable 3-manifold $M$ with empty or toroidal boundary, we have
\[LTV(M) \leq C ||M ||,\]
where $||M||$ denotes the Gromov norm (a.k.a. simplicial volume) of an oriented 3-manifold $M$.
\end{theorem}

\begin{corollary}[Corollary 5.3 \cite{detcherry2020gromov}]
\label{3}
    Let $M'$ be a compact, oriented 3-manifold with non-empty toroidal boundary, and let $M$ be a manifold obtained from $M'$ by Dehn filling some of the boundary components. Then,
    \[TV_r'(M)\leq TV_r'(M'),\] 
    and thus
    \[LTV(M) \leq LTV(M').\]
    
\end{corollary}

\section{Proof of Theorem \ref{main.A}} 
\label{section 3}
Let $M$ be an oriented Seifert fibered 3-manifold with boundary, described by the symbol \[(\epsilon,g; (a_1, b_1), \ldots, (a_n, b_n)).\] Let $\overline{M}$ be the same manifold $M$, but with the opposite orientation, represented by the symbol \[(\epsilon,g; (a_1, -b_1), \ldots, (a_n, -b_n)),\]
(Chapter 2 \cite{seifert89topology}). Then, the double $D(M)$ of $M$ is obtained by gluing $M$ and $\overline{M}$ together along their common boundary, $M\cup_{\partial M}\overline{M}$. Equivalently, $D(M)$ admits an orientation-reversing involution $i:D(M) \rightarrow {D(M)} $ such that $i(M)=\overline{M}$, $i(\overline{M})=M$ and $i_{|_{\partial M}}=\mathrm{id}$. In particular, gluing can be denoted as $D(M)=M\cup_{i_{|_{\partial M}}}\overline{M}$.

    Gluing $M$ and $\overline{M}$ along all boundary components produces a closed manifold. Since gluing preserves fibers, the exceptional fibers of $M$ remain intact in $D(M)$, and each is duplicated by the corresponding fiber in $\overline{M}$ but with a reversed orientation. Therefore, the total number of exceptional fibers in $D(M)$ is $2n$. At the level of 2-orbifold base, each boundary torus corresponds to a boundary circle on the base. Doubling $M$ corresponds to doubling the 2-orbifold base along its boundary circles, producing a closed 2-orbifold base whose underlying surface has genus $2g$. In particular, we can represent $D(M)$ by the symbol \[(\epsilon,2g; (a_1, b_1), \ldots, (a_n, b_n),(a_1, -b_1), \ldots, (a_n, -b_n)).\]

Before proving Theorem \ref{main.A}, we first prove Lemma \ref{lemma}. Here, we provide a subsequence for odd integer $r\geq3$ for which the lower bounds of the absolute value of Turaev-Viro invariants for oriented Seifert fibered 3-manifold with boundary $M$, and as well as for the double of $M$, are both strictly greater than 1 at $q=e^\frac{2\pi i}{r}$. Recall that, $b_j^*$ denotes any congruence inverse of $b_j$ modulo $a_j$ where $a_j\geq3$, namely $b_j b_j^* \equiv 1 (\text{mod }a_j)$.

\begin{lemma}
\label{lemma}

    Let $M$ be an oriented Seifert fibered 3-manifold described by the symbol 
\[
(\epsilon,g; (a_1, b_1), \ldots, (a_n, b_n)),
\] where $a_j\geq3$ for all $j=1,\ldots,n$ and let $D(M)$ be the double of $M$. Suppose that there is an integer $\gamma > 0$ and $\boldsymbol{\mu}=(\mu_1,\ldots ,\mu_n)$ with $\mu_j=\pm 1$ such that $\gamma+\mu_jb_j^* \equiv 0 (\textit{mod }a_j)$ for $j=1,\ldots, n$. Then, for $r$ divisible by $A:=\text{lcm}(a_1,\ldots ,a_n)$, \[|TV_r'(D(M),e^{\frac{2\pi i}{r}})|>1,\] and \[|TV_r'(M,e^{\frac{2\pi i}{r}})|>1.\]  
\end{lemma}

\begin{proof}We begin by computing the expression $Z_{(\epsilon,r)}(M,e^\frac{\pi i}{2r})$ in Lemma \ref{lemma seifert} for a closed, oriented Seifert fibered 3-manifold $D(M)$. In this computation, since the Euler number of $D(M)$ is zero, the exponential term $e^ \frac{i \pi e(M) \gamma^2}{2r}$ vanishes; likewise, if $A$ divides $r$, the exponential term $e^\frac{2\pi rm_j^2b_j^*}{a_j}$ can be simplified to the value 1. Additionally, the genus of $D(M)$ is twice that of $M$. We denote the computed expression for $D(M)$ as $Z_{(\epsilon,r)}(D(M),e^\frac{\pi i}{2r})$, where $(\gamma,\boldsymbol{\mu},\boldsymbol{m})$ ranges over $\{1,\ldots,r-1\}\times\{\pm1\}^{2n}\times \mathbb{Z}/a_1\mathbb{Z}\times \ldots \times \mathbb{Z}/a_n\mathbb{Z}\times \mathbb{Z}/a_1\mathbb{Z} \times \ldots \mathbb{Z}/a_n\mathbb{Z}$. 

 Notice that the exceptional fibers in $D(M)$ for $j=n+1,\ldots,2n$ are represented by $(a_j,-b_j)$. These are identical to the exceptional fibers for $j=1,\ldots,n$ but are negatively oriented. Consequently, the $j$-th product that ranges over $\{1,\ldots, 2n\}$ in $Z_{(\epsilon,r)}(D(M),e^\frac{\pi i}{2r})$ can be separated into two distinct products: one that ranges over $\{1,\ldots, n\}$, and another that ranges over $\{n+1,\ldots, 2n\}$. Thus, the sum $Z_{(\epsilon,r)}(D(M),e^\frac{\pi i}{2r})$ can be rearranged as follows.

\begin{equation}
\begin{split} Z_{(\epsilon ,r)}(D(M),e^\frac{\pi i}{2r})& = \sum_{(\gamma, \boldsymbol{\mu})} \left \{\frac{(-1)^{2\gamma a_\epsilon g}\prod_{j=1} ^{2n} \left(\mu_j e^{ {  \frac{ -i \pi \gamma\mu_j}{a_j r}}}\right)}{\mathrm{sin}^{2n+2a_\epsilon g-2}(\pi \gamma / r)}  \cdot \sum_{\boldsymbol{m}}\prod_{j=1}^{2n} e^ {-i \cdot \left(\frac{2 \pi m_j(\gamma + \mu_jb_j^*)}{a_j} \right)} \right \} \notag \\
                         &= \sum_{(\gamma, \boldsymbol{\mu})} \Bigg \{\frac{\prod_{j=1} ^n (\mu_j \mu_{n+j}) e^{-i \pi \gamma \cdot {\sum_{j=1}^n  \frac{ (\mu_j +\mu_{n+j})}{a_j r}}}}{\mathrm{sin}^{2n+2a_\epsilon g-2}(\pi \gamma / r)} \cdot\\  &\sum_{\boldsymbol{m}}\prod_{j=1}^n e^ {-i \cdot \left(\frac{2 \pi m_j(\gamma + \mu_jb_j^*)}{a_j} \right)} \cdot  e^ {-i \cdot \left(\frac{2 \pi m_j(\gamma - \mu_{n+j}b_j^*)}{a_j} \right)}\Bigg \} \notag \\
                         &= \sum_{(\gamma, \boldsymbol{\mu})} \Bigg \{\frac{\prod_{j=1} ^n (\mu_j \mu_{n+j}) e^{-i \pi \gamma \cdot {\sum_{j=1}^n  \frac{ (\mu_j +\mu_{n+j})}{a_j r}}}}{\mathrm{sin}^{2n+2a_\epsilon g-2}(\pi \gamma / r)}  \cdot \notag\\
                           \end{split} 
\end{equation}
 
 \begin{equation}
        \hspace{25pt}\prod_{j=1}^n\sum_{m_j} e^ {-i \cdot \left(\frac{2 \pi m_j(\gamma + \mu_jb_j^*)}{a_j} \right)} \cdot \prod_{j=1}^n\sum_{m_j} e^ {-i \cdot \left(\frac{2 \pi m_j(\gamma - \mu_{n+j}b_j^*)}{a_j} \right)}\Bigg \} \label{eq;3.1}   
 \end{equation}

For any fixed $(\gamma,\boldsymbol{\mu})=(\gamma,(\mu_1,\ldots,\mu_{2n}))$, we have  
\[\sum_{m_j\in {\mathbb{Z}/{a_j\mathbb{Z}}}} e^ {-i \cdot \left(\frac{2 \pi m_j(\gamma + \mu_jb_j^*)}{a_j} \right)} = \begin{cases}
    a_j & \text{if } \gamma + \mu_j b_j^* \equiv 0(\text{mod } a_j)  \\
    
    0 & \text{otherwise}
\end{cases} \text{    }   \]

and

\[\sum_{m_j\in {\mathbb{Z}/{a_j\mathbb{Z}}}} e^ {-i \cdot \left(\frac{2 \pi m_j(\gamma - \mu_{n+j}b_j^*)}{a_j} \right)} = \begin{cases}
    a_j & \text{if }  \gamma - \mu_{n+j} b_j^* \equiv0 (\text{mod } a_j) \\
    
    0 & \text{otherwise}
    
\end{cases}\\ \text{ }   \]for all $j=1,\ldots,n$. As a result, the summand corresponding to $(\gamma,\boldsymbol{\mu})=(\gamma,(\mu_1,\ldots,\mu_{2n}))$ in the simplified expression $Z_{(\epsilon ,r)}(D(M),e^\frac{\pi i}{2r})$ in \eqref{eq;3.1} is nonzero if and only if the following congruences hold for all $j=1,\ldots,n$:
\begin{equation}
    \gamma + \mu_j b_j^* \equiv 0(\text{mod } a_j)\label{eq;3.2}
\end{equation}
and
\begin{equation}
    \gamma - \mu_{n+j} b_j^* \equiv 0(\text{mod } a_j). \label{eq;3.3}
\end{equation}
In order for the sum $Z_{(\epsilon ,r)}(D(M),e^\frac{\pi i}{2r})$ to be nonzero, there must exist some $\gamma$ that satisfy the set of congruences \eqref{eq;3.2} and \eqref{eq;3.3} for some $\boldsymbol{\mu}$.

Recall that throughout, we assume $a_j \geq 3$ for all $j=1,\dots,n$. Subtracting \eqref{eq;3.2} from \eqref{eq;3.3}, we obtain
\[
(\mu_j+\mu_{n+j})\,b_j^* \equiv 0 (\text{mod } a_j).
\]
Multiplying both sides by $b_j$ yields
\[
\mu_j + \mu_{n+j} \equiv 0 (\text{mod } a_j).
\]
Because $a_j \neq 2$, this congruence forces the unique solution
\[
\mu_{n+j} \equiv -\mu_j (\text{mod } a_j),
\]
as required.

This means that the vectors of $\boldsymbol{\mu}$ which contribute to the existence of $\gamma$ must be of the form $(\mu_1,\ldots,\mu_n,-\mu_1,\ldots,-\mu_n)$. In other words, the expression $Z_{(\epsilon ,r)}(D(M),e^\frac{\pi i}{2r})$ is nonzero if and only if there exist a solution $\gamma\in \{1,\ldots,A-1\}$ for $\gamma + \mu_j b_j^* \equiv 0(\text{mod } a_j)$ for $j=1,\ldots,n$ where $\boldsymbol{\mu}$ are of the form $(\mu_1,\ldots,\mu_n,-\mu_1,\ldots,-\mu_n)$.

 Define $\mathcal{C}$ as the set of all vectors $\boldsymbol{\mu}$ of the form $(\mu_1,\ldots,\mu_n,-\mu_1,\ldots,-\mu_n)$. Let 
$\mathcal{B}$ be the set of all solutions for \(\gamma \in \{1,\ldots,A-1\}\) obtained by solving $\gamma + \mu_j b_j^* \equiv 0(\text{mod } a_j)$ for all $j=1,\ldots,n$ for each $\boldsymbol{\mu}\in \mathcal{C}$. That is,
\[\mathcal{B}=\{(\gamma,\boldsymbol{\mu}) | \gamma+\mu_jb_j^*\equiv 0 (\text{mod }a_j) \text{ for } j=1,\ldots,n\text{ and }\boldsymbol{\mu}\in \mathcal{C}  \}.\] 

By the above observations, we can simplify the sum $Z_{(\epsilon,r)}(D(M),e^\frac{\pi i}{2r})$ in \eqref{eq;3.1} furthermore as follows.

\[\begin{split}Z_{(\epsilon ,r)}(D(M),e^\frac{\pi i}{2r}) &= \sum_{\mathcal{B}} \left \{\frac{\prod_{j=1} ^n (\mu_j \mu_{n+j}) e^{-i \pi \gamma \cdot {\sum_{j=1}^n  \frac{ (\mu_j+ \mu_{n+j})}{a_j r}}}}{\mathrm{sin}^{2n+2a_\epsilon g-2}(\pi \gamma / r)}  \cdot \prod_{j=1}^n a_j \cdot \prod_{j=1}^n a_j\right\}\\
                   &=  \left(\prod_{j=1}^n a_j\right)^2 \cdot\sum_{\mathcal{B}} \left \{\frac{\prod_{j=1} ^n (-\mu_j ^2) e^{-i \pi \gamma \cdot {\sum_{j=1}^n  \frac{ (\mu_j -\mu_{j})}{a_j r}}}}{\mathrm{sin}^{2n+2a_\epsilon g-2}(\pi \gamma / r)}  \right\}\\
                    \end{split}\]
 \begin{equation}
     \hspace{-3pt}= \left(\prod_{j=1}^n a_j\right)^2 \cdot\sum_{\mathcal{B}} \left \{\frac{ (-1)^n }{\mathrm{sin}^{n+2a_\epsilon g-2}(\pi \gamma / r)}  \right\}\label{eq;3.4}
 \end{equation}

 According to the assumption, $r$ is divisible by $A$. In particular, if $r=A$, the simplified sum $Z_{(\epsilon,r=A)}(D(M),e^\frac{\pi i}{2A})$ in \eqref{eq;3.4} contains at least two nonzero summands. The corresponding $(\gamma,\boldsymbol{\mu)}$ pairs where $\gamma \in \{1,2,\ldots, A-1\}$ and $\boldsymbol{\mu }\in \mathcal{C}$ are as follows.

    \[(\gamma_1,\boldsymbol{\mu})=(\gamma,(\mu_1,\ldots,\mu_n,-\mu_1,\ldots,-\mu_n))\]
    and
    \[(\gamma_2,-\boldsymbol{\mu})=(A-\gamma,(-\mu_1,\ldots,-\mu_n,\mu_1,\ldots,\mu_n)).\]

Now, if $r=kA$ where $k$ is a positive integer, there exist at least $2k$ number of nonzero summands that contribute to the simplified sum $Z_{(\epsilon,r=kA)}(D(M),e^\frac{\pi i}{2kA})$ in \eqref{eq;3.4}. Therefore, the corresponding $(\gamma,\boldsymbol{\mu)}$ pairs are $(pA+\gamma,\boldsymbol{\mu})$, and $(pA+A-\gamma,-\boldsymbol{\mu})$ for $p=0,\ldots,k-1$, where $\gamma \in \{1,2,\ldots, A-1\}$ and $\boldsymbol{\mu }\in \mathcal{C}$. Consequently, the expression $Z_{(\epsilon,r)} (D(M),e^\frac{\pi i}{2r})$ (in \eqref{eq;3.4}) can be restated as  $Z_{(\epsilon ,kA)}(D(M),e^\frac{\pi i}{2kA}) $ for $r=kA$ where $\gamma \in \{1,2,\ldots, A-1\}$ and $\boldsymbol{\mu }\in \mathcal{C}$ as follows.

\begin{equation}
\begin{split}&Z_{(\epsilon ,kA)}(D(M),e^\frac{\pi i}{2kA})  =\left(\prod_{j=1}^n a_j\right)^2\cdot\sum_{\mathcal{B}}  \sum_{p=0}^{k-1} \left \{\frac{(-1)^n }{\mathrm{sin}^{2n+2a_\epsilon g-2}(\frac{\pi (pA+\gamma)}{kA} ) } +\frac{(-1)^n }{\mathrm{sin}^{2n+2a_\epsilon g-2}(\frac{\pi (pA+A-\gamma)}{kA})  }\right\}\\
                       &= (-1)^n\left(\prod_{j=1}^n a_j\right)^2\cdot \sum_{\mathcal{B}}\sum_{p=0}^{k-1}  \Bigg\{\frac{1 }{\mathrm{sin}^{2n+2a_\epsilon g-2}(\frac{\pi (pA+\gamma)}{kA} ) } 
                    +\frac{1}{\mathrm{sin}^{2n+2a_\epsilon g-2}(\frac{\pi (pA+A-\gamma)}{kA})  }\Bigg\}  \label{eq; 3.5}
\end{split}
\end{equation}
\\

Using Lemma \ref{lemma seifert}, we then compute the value of the corresponding $RT_r$ invariants for the 3-manifold $D(M)$ at 
$ e^{\frac{\pi i}{2r}}$ with $r = kA$ as follows:
\begin{equation}
     RT_{(\epsilon ,r=kA)}(D(M),e^\frac{\pi i}{2kA})=\frac{(-1)^{2a_\epsilon g}(kA)^{a_\epsilon g-1} i^{2n} Z_{(\epsilon,r=kA)}(D(M),e^{\frac{\pi i}{2kA}})}{2^{2n+a_\epsilon g-1}\sqrt{\left(\prod_{j=1}^n a_j\right)^2}}.  \label{eq 3.7}\end{equation}

Next, by Theorem \ref{TV WRT con B} and (\ref{eq 3.7}), the $TV_r$ invariant for $D(M)$ at $e^{\frac{\pi i}{r}}$ for $r=kA$ is given by

 \begin{equation}TV_{r=kA}(D(M),e^\frac{\pi i}{r})=||RT_{(\epsilon,r=kA)}(D(M),e^\frac{\pi i}{2r})||^2 = \left| \frac{(-1)^{2a_\epsilon g}(kA)^{a_\epsilon g-1} i^{2n} Z_{(\epsilon,r=kA)}(D(M),e^{\frac{\pi i}{2kA}})}{2^{2n+a_\epsilon g-1}\sqrt{\left(\prod_{j=1}^n a_j\right)^2}} \right|^2. \label{11}\end{equation}

However, since our focus is on odd integer $r\geq3$ and $q = e^{\frac{2\pi i}{r}}$, in particular when $r = kA$ for an odd positive integer $k$, we first use Lemma \ref{tv lemma} to compute the required $TV_r'$ invariant at $e^\frac{(r-1)\pi i}{r}$ as follows.

    \begin{equation}
        TV_{r=kA}'(D(M),e^\frac{(r-1)\pi i}{r})=\frac{TV_{r=kA}(D(M),e^\frac{\pi i}{r})}{TV_3(D(M),e^\frac{\pi i}{3})}
    \label{21}\end{equation}

Then, applying Galois conjugacy (transforming $e^{\frac{(r-1)\pi i}{r}} \mapsto e^{\frac{2\pi i}{r}}$ or $e^{\frac{\pi i}{r}} \mapsto e^{\frac{2\pi i}{r}}$) to (\refeq{21}) together with \eqref{11}, and using the fact that $TV_3(D(M),e^\frac{2\pi i}{3})=2^{b_2-b_0}$ for a closed 3-manifold $D(M)$, we obtain the $TV_r'$ invariant at $q = e^{\frac{2\pi i}{r}}$ as follows.

Note also that we must distinguish between the two possible transformations, depending on the parity of $r$:  
\[
\begin{cases}
e^{\frac{(r-1)\pi i}{r}} \;\mapsto\; e^{\frac{2\pi i}{r}}, & \text{if $r$ is odd (so $r-1$ is even)}, \\[6pt]
e^{\frac{\pi i}{r}} \;\mapsto\; e^{\frac{2\pi i}{r}}, & \text{if $r$ is odd (i.e.\ $r \equiv 1 \pmod{2}$)}.
\end{cases}
\]

\begin{equation}
    TV_{r=kA}'(D(M),q=e^\frac{2\pi i}{r})        = \left|\frac{(kA)^{a_\epsilon g-1}(-1)^nZ_{(\epsilon,r=kA)}(D(M),e^{\frac{\pi i}{kA}})}{2^{2n-a_\epsilon g-1}\cdot 2^{b_2-b_0}\left(\prod_{j=1}^n a_j\right)}\right|^2 \label{TV'dm}
   \end{equation}

\vspace{1pt}

   A similar argument applies to $M$: by Theorem \ref{TV WRT con A}, and since $\chi(M) = 0$, the $TV_r$ invariant of $M$ at $e^{\frac{\pi i}{r}}$ for $r = kA$ reduces to

 \begin{equation}
    TV_{r=kA}(M,e^\frac{\pi i}{r})=\eta^{-\chi(M)}RT_{(\epsilon,r=kA)}(D(M),e^{\frac{\pi i}{2r}})        = \frac{(-1)^{2a_\epsilon g}(kA)^{a_\epsilon g-1} i^{2n} Z_{(\epsilon,r=kA)}(D(M),e^{\frac{\pi i}{2kA}})}{2^{2n+a_\epsilon g-1}\sqrt{\left(\prod_{j=1}^n a_j\right)^2}}. 
   \label{12} \end{equation}

    Then, applying the same Galois conjugacy (transforming $e^{\frac{(r-1)\pi i}{r}} \mapsto e^{\frac{2\pi i}{r}}$ or $e^{\frac{\pi i}{r}} \mapsto e^{\frac{2\pi i}{r}}$) to (\ref{12}) and together with $$TV_{r=kA}'(M,e^\frac{(r-1)\pi i}{r})=\frac{TV_{r=kA}(M,e^\frac{\pi i}{r})}{TV_3(M,e^\frac{\pi i}{3})},$$ and since $TV_3(M,e^\frac{2\pi i}{3})=2^{b_2}$ for $M$ ($\partial M \neq \emptyset$), we get the $TV_r'$ invariant at $q = e^{\frac{2\pi i}{r}}$ as follows.
    
    \begin{equation}
        TV_{r=kA}'(M,q=e^\frac{2\pi i}{r})= \frac{(kA)^{a_\epsilon g-1}(-1)^nZ_{(\epsilon,r=kA)}(D(M),e^{\frac{\pi i}{kA}})}{2^{2n-a_\epsilon g-1}\cdot 2^{b_2}\left(\prod_{j=1}^n a_j\right)}  \label{TV'm}\end{equation}

Now, observe that the lower bounds of the absolute values of both \eqref{TV'dm} and \eqref{TV'm} depend on evaluating the lower bound of $|Z_{(\epsilon,r=kA)}(D(M),e^{\frac{\pi i}{kA}})|.$ 
\vspace{1pt}

Since $p$ and $\gamma$ ranges over $\{0,\ldots,k-1\}$ and $\{1,2,\ldots, A-1\}$ respectively, both integers $pA+\gamma$ and $pA+A-\gamma$ are less than $kA$.
Therefore, \[0<\frac{2\pi (pA+\gamma)}{kA},\frac{2\pi (pA+A-\gamma)}{kA}<2\pi.\] For $p=1,\ldots ,{(k-1)}/4$, the sine functions, $\text{sin}\left(\frac{2\pi (pA+\gamma)}{kA}\right)$, and $\text{sin}\left(\frac{2\pi (pA+A-\gamma)}{kA}\right)$ increase and take the values between 0 and 1. Conversely, for $p=(k-1)/4+1,\ldots ,(k-1)/2$, these sine functions decrease but still remain within the range between 0 and 1. A similar behavior occurs for $p=(k-1)/2+1,\ldots ,(k-1)$, but take values in between 0 and -1. However, since both sine functions are raised to even powers, their contributions are nonnegative, and it is enough to restrict attention to positive values. Therefore, this ensures that the reciprocals of $\text{sin}\left(\frac{2\pi (pA+\gamma)}{kA}\right)$, and $\text{sin}\left(\frac{2\pi (pA+A-\gamma)}{kA}\right)$ raised to the power $2n+2a_\epsilon g-2$ are strictly greater than 1 for $\gamma \in \{1,\ldots,A-1\}$ and $p=0,1,\ldots, k-1$. Also, let $|\mathcal{B}|$ be the number of elements in $\mathcal{B}$ for $r=kA$. Using these estimates, the lower bound of the absolute value of the expression $Z_{(\epsilon ,kA)}(D(M),e^\frac{\pi i}{kA})$ in \eqref{eq; 3.5} for $\gamma \in \{1,\ldots,A-1\}$ and $p=0,1,\ldots, k-1$ is computed as follows. 

\begin{equation}
\begin{split}|&Z_{(\epsilon,kA)}(D(M),e^{\frac{\pi i}{kA}})|\\   &=\Bigg|(-1)^n\left(\prod_{j=1}^n a_j\right)^2\sum_\mathcal{B}\sum_{p=0}^{k-1} \Bigg \{\frac{1 }{\mathrm{sin}^{2n+2a_\epsilon g-2}\left(\frac{2\pi (pA+\gamma)}{kA} \right) } +\frac{1}{\mathrm{sin}^{2n+2a_\epsilon g-2}\left(\frac{2\pi (pA+A-\gamma)}{kA}\right)  }\Bigg\}\Bigg|\\
                  &\geq\left|\left(\prod_{j=1}^n a_j\right)^2\cdot\sum_\mathcal{B}\sum_{p=0}^{k-1} \left (1   +1  \right)\right|\\
                  &= 2|\mathcal{B}|k\left(\prod_{j=1}^n a_j\right)^2  \\ 
                  \end{split}\label{eq 3.6}\end{equation}

Now, by using the lower bound estimate provided in \eqref{eq 3.6}, and the formula \eqref{TV'dm}, a lower bound for the absolute value of $TV_r'$ invariants of $D(M)$ at $e^\frac{2\pi i}{r}$ for $r=kA$ can be computed as follows.

 \[\begin{split}|TV_{r=kA}'(D(M),q=e^\frac{2\pi i}{kA})|
                                            &=\left|\frac{(kA)^{a_\epsilon g-1}(-1)^nZ_{(\epsilon,r=kA)}(D(M),e^{\frac{\pi i}{kA}})}{2^{2n-a_\epsilon g-1}\cdot 2^{b_2-b_0}\left(\prod_{j=1}^n a_j\right)} \right|^2\\
                                            &\geq \left|\frac{(kA)^{a_\epsilon g-1}2|\mathcal{B}|k \left(\prod_{j=1}^n a_j\right)}{2^{2n+a_\epsilon g-1}2^{b_2-b_0}}  \right|^2\\
                                            &>1.
  \end{split} \]
   The last inequality is true when $k$ is taken appropriately large.

Next, we compute the lower bound for the absolute value of the $TV_r$ invariants of $M$ with boundary at $q=e^\frac{2\pi i}{r}$ for $r=kA$ using the estimate \eqref{eq 3.6} and formula \eqref{TV'm}. Moreover, the last inequality is true for appropriately large $k$. 

\[\begin{split}|TV_{r=kA}'(M,q=e^\frac{2\pi i}{kA})|
                                            &=\left|\frac{(kA)^{a_\epsilon g-1}(-1)^nZ_{(\epsilon,r=kA)}(D(M),e^{\frac{\pi i}{kA}})}{2^{2n-a_\epsilon g-1}\cdot 2^b\left(\prod_{j=1}^n a_j\right)} \right|\\
                                            &\geq \left|\frac{(kA)^{a_\epsilon g-1}2|\mathcal{B}|k \left(\prod_{j=1}^n a_j\right)}{2^{2n+a_\epsilon g-1}2^{b_2}}  \right|\\
                                            &>1.
  \end{split} \]

\end{proof}

We now prove Theorem \ref{main.A} of Introduction.\\

\newcommand{\theoremname}{Theorem \ref{main.A}}
\begin{theorem-rep}
Let $M$ be an oriented Seifert fibered 3-manifold with boundary, described by the symbol 
\[
(\epsilon,g; (a_1, b_1), \ldots, (a_n, b_n)).
\] where each $a_j\geq3$ and gcd$(a_j,b_j)=1$ for all $j=1\ldots,n$. Suppose that there is an integer $\gamma > 0$ and $\boldsymbol{\mu}=(\mu_1,\ldots ,\mu_n)$ with $\mu_j=\pm 1$ such that \[\gamma+\mu_jb_j^* \equiv 0 (\text{mod }a_j)\] for $j=1,\ldots,n$.
Then, $M$ and $D(M)$ satisfy Conjecture \ref{con:TV volume}.

That is, $LTV(M)=LTV(D(M))=0$.
\end{theorem-rep}

\begin{proof}Suppose there is an integer $\gamma > 0$ and $\boldsymbol{\mu}=(\mu_1,\ldots ,\mu_n)$ with $\mu_j=\pm 1$ such that $\gamma+\mu_jb_j^* \equiv 0 (\text{mod }a_j)$ for $j=1,\ldots, n$. Then by Lemma \ref{lemma}, for $r$ divisible by $A=\text{lcm}(a_1,\ldots ,a_n)$, in particular, for $r=kA$ where $k$ ranges over odd positive integers,

\[|TV_{kA}'(D(M),q=e^\frac{2\pi i}{kA})| \geq \left|\frac{(kA)^{a_\epsilon g-1}2|\mathcal{B}|k \left(\prod_{j=1}^n a_j\right)}{2^{2n+a_\epsilon g-1}2^{b_2-b_0}}  \right|^2 >1,\] and 
\[|TV_{kA}'(M,q=e^\frac{2\pi i}{kA})| \geq \left|\frac{(kA)^{a_\epsilon g-1}2|\mathcal{B}|k \left(\prod_{j=1}^n a_j\right)}{2^{2n+a_\epsilon g-1}2^{b_2}}  \right| >1.\]
                                            
Therefore, for $k$ that runs over all odd positive integers, we have

\[\begin{split}LTV(D(M))&=\underset{k \xrightarrow{}\infty}{\text{lim sup}}\frac{2\pi}{kA}\text{log}|TV_{r=kA}'(D(M),q)|\\
                        &\geq \underset{k \xrightarrow{}\infty}{\text{lim sup}}\frac{2\pi}{kA}\text{log}\left|\frac{(kA)^{a_\epsilon g-1}2|\mathcal{B}|k \left(\prod_{j=1}^n a_j\right)}{2^{2n+a_\epsilon g-1}2^{b_2-b_0}}  \right|^2\\
                        &=0.\end{split}\]
Similarly, \[LTV(M)\geq \underset{k \xrightarrow{}\infty}{\text{lim sup}}\frac{2\pi}{kA}\text{log}\left|\frac{(kA)^{a_\epsilon g-1}2|\mathcal{B}|k \left(\prod_{j=1}^n a_j\right)}{2^{2n+a_\epsilon g-1}2^{b_2}}  \right|=0.\] Since the Gromov norm of Seifert fibered 3-manifolds is zero, we have $\|M\|=\|D(M)\|=0 $. By Theorem \ref{1}, it follows that \[LTV(D(M)),LTV(M)\leq 0.\] 
  Therefore,\[LTV(M)=LTV(D(M))=v_3||D(M)||=v_3||M||=0,\] where $k$ runs over all odd integers.

Thus, both $M$ and $D(M)$ satisfy Conjecture \ref{con:TV volume}.\\

\end{proof}

\section{Some examples}
\label{section 4}

In this section, we discuss particular families of oriented Seifert fibered 3-manifolds where Theorem \ref{main.A} applies.

\renewcommand{\theoremname}{Corollary \ref{coro A}}
\begin{corollary*} Let $M$ be an oriented Seifert fibered 3-manifold with boundary, described by the symbol 
\[
(\epsilon,g;(a_1, b_1), \ldots, (a_n, b_n))
,\] where each $a_j\geq3$ and gcd$(a_j,b_j)=1$ for all $j=1\ldots,n$. If either
\begin{enumerate}
    \item [(a)] $a_1,\ldots , a_n$ are coprime; or
    \item [(b)] $a_1=\ldots =a_n=a$ and there are $\mu_1,\ldots,\mu_n$ such that $\mu_1 b_1^* \equiv \ldots \equiv \mu_n b_n^* \pmod{a}$,
\end{enumerate} then $M$ and $D(M)$ satisfy Conjecture \ref{con:TV volume}.
\end{corollary*}

    \begin{proof}Recall that, $b_j^*$ denotes any congruence inverse of $b_j$ modulo $a_j$; that is $b_j b_j^* \equiv 1 (\text{mod }a_j)$, $\mathcal{C}$ is the set of all vectors $\boldsymbol{\mu}$ of the form $(\mu_1,\ldots,\mu_n,-\mu_1,\ldots,-\mu_n)$,
and \[\mathcal{B}=\{(\gamma,\boldsymbol{\mu}) | \gamma+\mu_jb_j^*\equiv 0(\text{mod }a_j) \text{ for } j=1,\ldots,n\text{ and }\boldsymbol{\mu}\in \mathcal{C}  \}.\]  

Suppose $a_1,\ldots, a_n$ are coprime to each other. Then, by using Chinese remainder theorem, there exists a unique solution $\gamma$ modulo $\text{lcm}(a_1,\ldots,a_n)=\prod_{j=1}^na_j$ for each $\boldsymbol{\mu}\in \mathcal{C}$ for the following set of congruence equations
    \[
    \gamma + \mu_j b_j^* \equiv0(\text{mod }{a_j}) \quad \text{for all } j = 1, \ldots, n.
    \]
    Therefore, we will have $|\mathcal{B}|=2^n$ different $\gamma$ values for each $\boldsymbol{\mu}\in \mathcal{C}$.\\

     Next, suppose $a_1=\ldots= a_n=a$. Then, \( \gamma \) for $\boldsymbol{\mu}\in \mathcal{C}$ exists only if the following congruences are satisfied:
    \[
    \mu_1 b_1^* \equiv \ldots \equiv \mu_n b_n^* \pmod{a}.\]

Then, by Theorem \ref{main.A}, $M$ and $D(M)$ satisfy the Conjecture \ref{con:TV volume}.
    
\end{proof}

 \begin{remark}
    An example case where $a_1=\ldots= a_n=a$ is when $b_1=\ldots=b_n=b$. Then, there exist $\gamma$ for $\boldsymbol{\mu}=(1,\ldots,1,-1,\ldots,-1)$ and $\boldsymbol{\mu}=(-1,\ldots,-1,1,\ldots,1)$, thus satisfy Conjecture \ref{con:TV volume}.

    Moreover, for any positive integer $a_j$ for $j=1,\ldots,n$, a solution $\gamma$ modulo $\text{lcm}(a_1,\ldots,a_n)$ exists if and only for every pair $(\mu_sb_s^*, \mu_tb_t^*)$ the congruence equation $\mu_sb_s^*\equiv \mu_tb_t^*(\text{mod gcd}(a_s,a_t))$ is satisfied where $s,t\in\{1,\ldots, n\}$. In particular, if $b_1=\ldots=b_n=b$, we have $\mu_sb^*\equiv \mu_tb^*(\text{mod gcd}(a_s,a_t))$ for $\boldsymbol{\mu}=(1,\ldots,1,-1,\ldots,-1)$ and $\boldsymbol{\mu}=(-1,\ldots,-1,1,\ldots,1)$.
 \end{remark}

Now, we prove Corollary \ref{exam 3}.

\renewcommand{\theoremname}{Corollary \ref{exam 3}}
\begin{corollary*}

 Suppose $M$ and $D(M)$ are 3-manifolds in the statement of Theorem \ref{main.A}. Let $L$ be a link in $M$ or $D(M)$ that has simplicial volume (a.k.a. Gromov norm) zero. Then, Conjecture \ref{con:TV volume} is true for $M\backslash L$.

\end{corollary*}

\begin{proof}According to Theorem \ref{main.A}, we have \[LTV(D(M))=LTV(M) = 0.\]
Next, by applying Corollary \ref{3} to 3-manifold $M \backslash L$, \[TV_r(M)\leq TV_r(M \backslash  L ),\] and since the Gromov norm $||M\backslash L||=0$, we get the following by using Corollary \ref{2}. \[LTV(M \backslash L) \leq 0.\] 

    Therefore,
    \[0=LTV(M) \leq LTV(M \backslash L) \leq 0.\]

    A similar argument holds for $D(M)$.
\end{proof}

\bibliographystyle{plainnat}  
\bibliography{references}  

\newpage
\hspace{-16pt}Shashini Marasinghe\\
Department of Mathematics, Michigan State University\\
East Lansing, MI, 48824\\
(marasin1@msu.edu)

\end{document}